 \documentclass[final,1p,times]{elsarticle}
\usepackage{amssymb}
\usepackage{amsthm}
\usepackage{hyperref}
\usepackage{amsmath,amssymb,amsopn,amsfonts,mathrsfs,amsbsy,amscd}
\usepackage{longtable}
\usepackage{color}
\journal{??}

\newcommand{\e}{\epsilon}

\newtheorem{definition}{Definition}[section]
\newtheorem{theorem}{Theorem}[section]

\newtheorem{lemma}{Lemma}[section]
\newtheorem{corollary}{Corollary}[section]

\numberwithin{equation}{section}

\begin{document}
\begin{frontmatter}
\title{Local, 2-local derivations and biderivations on 3-parameter generalized quaternion }

\author[label1]{Hassan Oubba}
%\author[label2]{Hassan Oubba}
\address[label1]{Universit\'e Moulay Isma\"{i}l\\
		Facult\'e des sciences -  D\'epartement de Mth\'ematiques\\
		B.P. 11201, Zitoune, M\'ekn\`es 50000, Maroc\\e-mail:{ hassan.oubba@edu.umi.ac.ma}}
	%\address[label2]{Universit\'e Moulay Isma\"{i}l\\
		%Facult\'e des sciences -  D\'epartement de Mth\'ematiques\\
		%B.P. 11201, Zitoune, M\'ekn\`es 50000, Maroc\\e-mail:
		 %hassan.oubba@edu.umi.ac.ma
	%}

\begin{abstract}
This article investigates the recently introduced three-parameter generalized quaternion algebra (3PGQ), denoted here as $\mathbb{K}_{\lambda_1,\lambda_2,\lambda_3}$
 . Our analysis is structured in three parts. First, we demonstrate that every local and 2-local derivation on this algebra is automatically a derivation. Second, we provide a complete characterization of its biderivations. Finally, we describe its commuting maps and centroid.
\end{abstract}
\end{frontmatter}

\section*{Introduction} \label{Introduction}
 The study of quaternions, initiated by Sir William Rowan Hamilton (1805-1865), provides a critical generalization of complex numbers. Quaternion algebra occupies a central position in modern mathematics, intersecting with diverse fields such as non-commutative ring theory, Lie theory, geometry, and number theory. Motivated by these connections, this work seeks to enrich the repertoire of the quaternion algebra $\mathbb{K}_{\lambda_1,\lambda_2,\lambda_3}$ by establishing new properties. The exploration of local derivations and biderivations is a well-established line of inquiry for algebraic structures like Lie algebras and rings, as these maps offer powerful algebraic indicators for probing their underlying structures.
 
Let $\mathcal{A}$ be an algebra (not necessary associative). Recall that a linear map $D :\mathcal{A} \rightarrow \mathcal{A}$ is said to be a derivation, if $D(ab)=D(a)b+aD(b)$ for all $a,b \in \mathcal{A}$. A linear map $\Delta : \mathcal{A} \rightarrow \mathcal{A}$ is called local derivation if for every $x \in \mathcal{A}$ there exists a derivation $D$ of $\mathcal{A}$ such that $\delta(x)=D_x(x)$. A 2-local derivation is a map (not necessarily linear) $\Delta : \mathcal{A} \rightarrow \mathcal{A}$ such that, for all $x,y \in \mathcal{A}$ there exists a derivation $D_{x,y}$ of $\mathcal{A}$ satisfies $\Delta(x)=D_{x,y}(x)$ and $\Delta(y)=D_{x,y}(y)$. The study of local and 2-local derivations of non-associative algebras was initiated in some papers of Ayuov and Kudaybergenov (for the case of Lie algebras, see \cite{Ayo1, Ayo2}. In particular, they proved that there are no non-trivial local and 2-local derivations on semisimple finite-dimensional Lie Algebras. In \cite{Ayo4} examples of 2-local derivations on nilpotent Lie algebras which are not derivation, were also given. Later, the study of local and 2-local derivations was continued for Leibniz algebras \cite{Ayo3}.\\
 A bilinear map $ \phi : \mathcal{A} \times \mathcal{A} \rightarrow \mathcal{A}$ is said to be biderivation, if $\phi(ab,c)=\phi(a,c)b+a\phi(b,c)$ and $ \phi(a,bc)=\phi(a,b)c+b\phi(a,c)$ for all $a,b,c \in \mathcal{A}$.  That means for all $a ,b\in \mathcal{A}$ the two linear maps $\phi(a,.): \mathcal{A} \rightarrow \mathcal{A}$, $c \rightarrow \phi(a,c)$ and $\phi(.,b): \mathcal{A} \rightarrow \mathcal{A}$, $c \rightarrow \phi(c,b)$, are derivations of $\mathcal{A}$. \\
Commuting maps and biderivations arose first in the associative ring theory \cite{Bres1, BresMar2}. Then, many authors have made considerable efforts to make the study of these maps very successful, see for example \cite{bn, bbo, Chen, oubs, Gua, Han, Leger, oubm, oubw, oubd, Tang, Wang1, Wang2}. To study biderivations and commuting linear maps of Schr\"odininger-Virasoro Lie algebra, in \cite{Wang1}, the authors use the $\mathbb{Z}$-graduation of this algebra. Furthermore, in \cite{BresZhao} the authors give a general method to characterize biderivations and commuting linear maps for a large class of Lie algebras, the results obtained show under certain conditions the crucial relationship between biderivations and commuting linear maps of a Lie algebra $\mathfrak{g}$ and the elements of $cent(M)$, with $cent(M)$ denoted the centroid of a $\mathfrak{g}$-module. The way used in \cite{Tang} requires the use of root systems of the simple Lie algebra. In this paper, using a classification  theorem in \cite{cha} and a computational method, we prove that any local and 2-local derivations of $\mathbb{K}_{\lambda_1,\lambda_2,\lambda_3}$ ($\lambda_3\neq0$) are derivation(Theorem \ref{thm2} ,\ref{thm3}). Also we characterize all biderivations, commuting linear maps and centroid of generalized quaternion algebras. Our first main result (Theorem \ref{thm2}, \ref{thm3}) is to prove that, every local (2-local) derivation of a 3-parameter generalized quaternion is a derivation. In section 2 we state our second main result as the following: Let $\mathbb{K}_{\lambda_1,\lambda_2, \lambda_3\neq0}$, be a 3-parameter generalized quaternion over $\mathbb{R}$. Then, $\delta$ is a  biderivation of $\mathbb{K}_{\lambda_1,\lambda_2, \lambda_3}$ if and only if there is a real number $\mu$ such that
$$\delta(x,y)=\mu \begin{vmatrix}
		\frac{\lambda_3}{\lambda_1} e_1&\frac{\lambda_2}{\lambda_1} e_2&e_3\\
		x_1&x_2&x_3 \\
		y_1&y_2&y_3
		\end{vmatrix}, \hspace{0.2 cm} x,y \in \mathbb{K}_{\lambda_1,\lambda_2, \lambda_3 }.$$
Our second third result (Theorem \ref{thm6}) is to characterize the form of  commuting linear maps of 3-parameter generalized quaternion. The last main result to (Theorem \ref{thm8}) prove that the centroid of a 3-parameter generalized quaternion is a field.\\
 This paper is organized as follows. In the first section we recall the necessary information concerning 3-parameter generalized quaternion. In Section 2 we prove that any local and 2-local derivation on 3-parameter generalized quaternion is a derivation. Section 3 is dedicated to characterize the form of all biderivations on a 3-parameter generalized quaternion. The last section aims to characterize the form of commuting linear maps and the centroid of a 3-parameter generalized quaternion. 
\section{Preliminaries}
In this paper, we consider  3-parameter generalized quaternion $\mathbb{K}_{\lambda_1 ,\lambda_2, \lambda_3}$ over $\mathbb{R}$ with assumption $\lambda_3 \neq 0$.
Quaternions were invented by Sir William Rowan Hamiltonian as an extension to the complex numbers. Hamiltonian’s defining relation is most succinctly written as:
  $$i^2=j^2=k^2=-1, \hspace{0.3 cm}ijk=-1.$$
  Split quaternions are given by 
  $$i^2=-1, \hspace{0.3 cm}j^2=k^2=1, \hspace{0.3 cm}ijk=1.$$
 Generalized quaternions are  defined as
 $$i^2=-\alpha, \hspace{0.3 cm} j^2=-\beta, \hspace{0.3 cm}k^2=-\alpha \beta, \hspace{0.3 cm}ijk=-\alpha \beta.$$
On the other hand, T. D. \c{S}ent\"urk and Z. \"Unal have introduced a new type
of quaternion family, which is called the 3-parameter generalized quaternion
(shortly, 3PGQ) in~\cite{sen}. In order to achieve a generalization of real, split,
and 2PGQ, the authors derive a comprehensive understanding of the quaternion
algebra based on the 3-parameters. The set of 3PGQs is denoted by $\mathbb{K}_{\lambda_1,\lambda_2,\lambda_3}$ and
defined as
\[
\mathbb{K}_{\lambda_1,\lambda_2,\lambda_3} = \lbrace a + b e_1 + c e_2 + d e_3 \mid a,b,c,d,\lambda_1, \lambda_2, \lambda_3 \in \mathbb{R},\ 
e_1^2 = -\lambda_1\lambda_2,\ e_2^2 = -\lambda_1\lambda_3,\ e_3^2 = -\lambda_2\lambda_3,\ e_1 e_2 e_3 = -\lambda_1\lambda_2\lambda_3 \rbrace.
\]
Each element $x= x_0 + x_1 e_1 + x_2 e_2 + x_3 e_3$ of the set $k_{\lambda_1,\lambda_2,\lambda_3}$ is called a \emph{3-parameters generalized quaternion (3PGQ)}. The real numbers $x_0, x_1, x_2, x_3$ are called components of $x$. Note that $\mathcal{B}(k_{\lambda_1,\lambda_2,\lambda_3}) = \{e_0, e_1, e_2, e_3\}$ is the base of the $\mathbb{K}_{\lambda_1,\lambda_2,\lambda_3}$, its vectors verify the following multiplication table:

\[
\begin{array}{c|cccc}
 & e_0 & e_1 & e_2 & e_3 \\
\hline
e_0 & 1 & e_1 & e_2 & e_3 \\
e_1 & e_1 & -\lambda_1\lambda_2 & \lambda_1 e_3 & -\lambda_2 e_2 \\
e_2 & e_2 & -\lambda_1 e_3 & -\lambda_1\lambda_3 & \lambda_3 e_1 \\
e_3 & e_3 & \lambda_2 e_2 & -\lambda_3 e_1 & -\lambda_2\lambda_3 \\
\end{array}
\]

\noindent \textbf{Special cases:}
\begin{enumerate}
    \item[(i)] If $\lambda_1=1$, $\lambda_2=\alpha$, $\lambda_3=\beta$, then the algebra of $2PGQs$ is obtained.
    \item[(ii)] If $\lambda_1 = 1$, $\lambda_2 = 1$, $\lambda_3 = 1$, then the algebra of Hamilton quaternions is achieved.
    \item[(iii)] If $\lambda_1 = 1$, $\lambda_2 = 1$, $\lambda_3 = -1$, then gives us the algebra of split quaternions.
    \item[(iv)] If $\lambda_1 = 1$, $\lambda_2 = 1$, $\lambda_3 = 0$, then the algebra of semi-quaternions is attained.
    \item[(v)] If $\lambda_1 = 1$, $\lambda_2 = -1$, $\lambda_3 = 0$, then we get the algebra of split semi-quaternions.
\end{enumerate}

Any 3PGQ, $x = x_0 + x_1 e_1 + x_2 e_2 + x_3 e_3$ consists of two parts, the scalar part and the vector part: $x = x_0 + \tilde{x}$ such as:
\[
 \tilde{x} = x_1 e_1 + x_2 e_2 + x_3 e_3.
\]

The rules of addition, scalar multiplication and multiplication are defined on $\mathbb{K}_{\lambda_1,\lambda_2,\lambda_3}$ as follows:

Let $x = x_0 + x_1 e_1 + x_2 e_2 + x_3 e_3$, $y = y_0 + y_1 e_1 + y_2 e_2 + y_3 e_3$ be 3PGQ, and $c$ be a real number.\\
 $\bullet$\text{Addition:} 
$$x + y = (x_0 + y_0) + (\tilde{x} + \tilde{y}) 
\quad = (x_0 + y_0) + (x_1 + y_1)e_1 + (x_2 + y_2)e_2 + (x_3 + y_3)e_3.$$
$\bullet$ \text{Multiplication by scalar:} 
$$c x= c x_0 + c x_1 e_1 + c x_2 e_2 + c x_3 e_3 \quad \text{for all } c \in \mathbb{R}.$$
$ \bullet$\text{Multiplication: from the multiplication table we have} 
 \begin{eqnarray*}
xy& =& (x_0 y_0 - \lambda_1 \lambda_2 x_1 y_1 - \lambda_1 \lambda_3 x_2 y_2 - \lambda_2 \lambda_3 x_3 y_3) \\
&& + e_1(x_0 y_1 + y_0 x_1 + \lambda_3 x_2 y_3 - \lambda_3 x_3 y_2) \\
&& + e_2(x_0 y_2 + y_0 x_2 + \lambda_2 x_3 y_1 - \lambda_2 x_1 y_3) \\
&& + e_3(x_0 y_3 + y_0 x_3 + \lambda_1 x_1 y_2 - \lambda_1 x_2 y_1). 
\end{eqnarray*}
Then, the multiplication can exprime as follows:
\[
xy = x_0 y_0 - f(\tilde{y}, \tilde{x}) + \tilde{y}\tilde{x} + \tilde{x}\tilde{y} + \tilde{y} \wedge \tilde{x},
\]
where
\[
f(\tilde{y}, \tilde{x}) = \lambda_1 \lambda_2 y_1 x_1 + \lambda_1 \lambda_3 y_2 x_2 + \lambda_2 \lambda_3 y_3 x_3
\]
and
\[
\tilde{y} \wedge \tilde{x} = 
\begin{vmatrix}
\lambda_3 \mathbf{e}_1 & \lambda_2 \mathbf{e}_2 & \lambda_1 \mathbf{e}_3 \\
y_1 & y_2 & y_3 \\
x_1 & x_2 & x_3
\end{vmatrix}.
\]
Here $\tilde{y} \wedge \tilde{x} = \lambda_3(y_2 x_3 - y_3 x_2)\mathbf{e}_1 + \lambda_2(y_3 x_1 - y_1 x_3)\mathbf{e}_2 + \lambda_1(y_1 x_2 - x_2 x_1)\mathbf{e}_3$.

\text{It's worth noting that } $e_0$ \text{ serves as an identity, which implies that } $e_0 \cdot e_i = e_i \cdot e_0 = e_i$ \text{ for every } $i$, 
\text{and so the center of } $k_{\lambda_1,\lambda_2,\lambda_3} \text{ is } Z(k_{\lambda_1,\lambda_2,\lambda_3}) = \mathbb{R} \cdot e_0 = \mathbb{R}$.

\section{Local and 2-Local Derivation of  3-parameter generalized quaternion}
We start by recall definition of derivation and related maps on arbitrary algebra.
\begin{definition}
\label{def2}
A derivation of an algebra $(\mathcal{A},.)$ is a linear map $d: \mathcal{A} \rightarrow \mathcal{A}$ such that 
$$d(x.y)=d(x).y+x.d(y) \hspace{1 cm}(1.1)$$
for all $x,y \in \mathcal{A}$.
\end{definition}
It is clear that the set of all derivations of an algebra $\mathcal{A}$ forms a vector space, which we denote by $Der(\mathcal{A})$. Recall that $Der(\mathcal{A})$ is a sub Lie algebra of $\mathfrak{gl}(\mathcal{A})$.
\begin{definition}
\label{def3}
Let $(\mathcal{A},.)$ be an algebra. An linear map $\Delta : \mathcal{A} \rightarrow \mathcal{A}$ is called local derivation, if for any element $x \in \mathcal{A}$ there exists a derivation $D_x : \mathcal{A} \rightarrow \mathcal{A}$ such that $\Delta(x)=D_x(x)$.
\end{definition}
\begin{definition}
\label{def4}
Let $(\mathcal{A},.)$ be an algebra. A (not necessarily linear) map $\Delta : \mathcal{A} \rightarrow \mathcal{A}$ is called 2-local derivation, if for any elements $x,y \in \mathcal{A}$ there exists a derivation
$D_{x,y} : \mathcal{A} \rightarrow \mathcal{A}$ such that $\Delta(x)=D_{x,y}(x)$ and $\Delta(y)=D_{x,y}(y)$.
\end{definition}
The following theorem recently has been proved by R. Chaker and A. Boua \cite{cha} and presents the form of any derivation of 3-parameter generalized quaternion.  
\begin{theorem} \cite{cha}
\label{thm1}
Let $d$ be a derivation of $k_{\lambda_1,\lambda_2,\lambda_3}$. Then, the matrix $D_d$ of $d$ is of the form:
\begin{equation} 
\label{eq1}
D_d= \begin{pmatrix}
		0 & 0 &0 &0\\
		0&0&\frac{-\lambda_3}{\lambda_2}a & -\frac{\lambda_3}{\lambda_1} b \\
		0 & a & d & -\frac{\lambda_2}{\lambda_1} c \\
		0 & b & c & d  
		\end{pmatrix}, 
\end{equation}
where $a,b,c,d\in \mathbb{R}, \lambda_1\lambda_2\neq0$ such that
\[
 \begin{cases}
d=d_{\lambda_1, \lambda_3}\neq0, & \text{if } \lambda_1 \lambda_3 = 0 \\
d=0, & \text{otherwise}
\end{cases}
\]
\end{theorem}
Now, we state and prove our first main result
\begin{theorem}
\label{thm2}
Every local derivation of the algebra $k_{\lambda_1,\lambda_2,\lambda_3}$ ($\lambda_3\neq0)$ is a derivation.
\end{theorem}
\begin{proof}
 Let $\Delta$ be an arbitrary local derivation of $k_{\lambda_1,\lambda_2,\lambda_3}$ and write
$$\Delta(x)=BX, \hspace{0.3 cm} x \in \mathcal{H}_d,$$
where $B=(b_{ij})_{0 \leq i,j \leq 3}$, $X=(x_0,x_1,x_2,x_3)$ is the vector corresponding to x. Then for every $x \in k_{\lambda_1,\lambda_2,\lambda_3}$ there exist elements $a^x,b^x,c^x \in \mathbb{R}$ such that
$$BX=\begin{pmatrix}
		0 & 0 &0 &0\\
		0&0&\frac{-\lambda_3}{\lambda_2}a^x & -\frac{\lambda_3}{\lambda_1} b^x \\
		0 & a^x & 0 & -\frac{\lambda_2}{\lambda_1} c^x \\
		0 & b^x & c^x & 0 
		\end{pmatrix} 
\begin{pmatrix}
x_0\\
x_1\\
x_2\\
x_3
\end{pmatrix}$$
In other words 
\begin{gather*}
		  	 \begin{cases}
		  	b_{00}x_0+b_{01}x_1+b_{02}x_2+b_{03}x_3=0 ;&\\
		  	\,b_{10}x_0+b_{11}x_1+b_{12}x_2+b_{13}x_3=-\frac{\lambda_3}{\lambda_2}a^x x_2 -\frac{\lambda_3}{\lambda_1} b^x x_3; \\
		  	\,b_{20}x_0+b_{21}x_1+b_{22}x_2+b_{23}x_3=a^x x_1 -\frac{\lambda_2}{\lambda_1} c^x x_3; &\\	
		  	\, b_{30}x_0+b_{31}x_1+b_{32}x_2+b_{33}x_3=b^x x_1+c^x x_2 .&\\ 	  	 
		  	 \end{cases}
\end{gather*}
Taking $x=e_0=(1,0,0,0)$, we have
\begin{gather*}
		  	 \begin{cases}
		  	b_{00}=0 ;&\\
		  	\,b_{10}=0; \\
		  	\,b_{20}=0;&\\	
		  	\, b_{30}=0.&\\ 	  	 
		  	 \end{cases}
\end{gather*}
For $x=e_1=(0,1,0,0)$, we get
\begin{gather*}
		  	 \begin{cases}
		  	b_{01}=0 ;&\\
		  	\,b_{11}=0 ; \\
		  	\,b_{21}=a^{e_1}  ;&\\	
		  	\, b_{31}=b^{e_1} .&\\ 	  	 
		  	 \end{cases}
\end{gather*}
Taking also $x=e_2=(0,0,1,0)$, we deduce
\begin{gather*}
		  	 \begin{cases}
		  	b_{02}=0 ;&\\
		  	\,b_{12}=-\frac{\lambda_3}{\lambda_2}a^{e_2} ; \\
		  	\,b_{22}=0  ;&\\	
		  	\,b_{32}=c^{e_2} .&\\ 	  	 
		  	 \end{cases}
\end{gather*}
Finally taking $x=e_3=(0,0,0,1)$, we get
\begin{gather*}
		  	 \begin{cases}
		  	b_{03}=0 ;&\\
		  	\,b_{13}=-\frac{\lambda_3}{\lambda_1} b^{e_3} ; \\
		  	\,b_{23}=-\frac{\lambda_2}{\lambda_1} c^{e_3} ;&\\	
		  	\, b_{33}=0 .&\\ 	  	 
		  	 \end{cases}
\end{gather*}
These equalities show that the matrix of the linear map $\Delta$ is
$$B=\begin{pmatrix}
		0 & 0 &0 &0\\
		0&0&\frac{-\lambda_3}{\lambda_2}a^{e_2} & -\frac{\lambda_3}{\lambda_1} b^{e_3} \\
		0 & a^{e_1} & 0 & -\frac{\lambda_2}{\lambda_1} c^{e_3} \\
		0 & b^{e_1} & c^{e_2} & 0 
		\end{pmatrix} $$
The equalities $\Delta(e_1+e_2)=\Delta(e_1)+\Delta(e_2)$ which implies $$a^{e_1}=a^{e_2}.$$
From the equalities $\Delta(e_1+e_3)=\Delta(e_1)+\Delta(e_3)$ we have $$b^{e_1}=b^{e_2}.$$
From the equalities $\Delta(e_2+e_3)=\Delta(e_2)+\Delta(e_3)$ we have $$c^{e_2}=c^{e_3}.$$
Therefore, 
$$B=\begin{pmatrix}
		0 & 0 &0 &0\\
		0&0&\frac{-\lambda_3}{\lambda_2}a^{e_1} & -\frac{\lambda_3}{\lambda_1} b^{e_1} \\
		0 & a^{e_1} & 0 & -\frac{\lambda_2}{\lambda_1} c^{e_2} \\
		0 & b^{e_1} & c^{e_2} & 0 
		\end{pmatrix} $$
 Therefore, by Theorem \ref{thm1}, $\Delta$ is a derivation.
\end{proof}
\textbf{Question} What happens if $\lambda_3=0$?\\
In the case $\lambda_3=0$, the linearity of $\Delta$ does not imply, for example $a^{e_1}=a^{e_2}$.  So, probably the case of $\lambda_3=0$, there exists a local derivation wich is not a derivation.\\
Now, we prove that any 2-Local derivation is a derivation.
\begin{theorem}
\label{thm3}
Every 2-local derivation of the algebras $k_{\lambda_1,\lambda_2,\lambda_3}$ $(\lambda_3\neq0)$ is a derivation.
\end{theorem}
\begin{proof}
Let $\Delta$ be a 2-local derivation of $k_{\lambda_1,\lambda_2,\lambda_3}$ and $i=0,1,2,3$. Then, by definition, for every element $x \in \mathbb{K}_{\lambda_1,\lambda_2, \lambda_3}$, there exists a derivation $D_{x,e_i}$ of $\mathbb{K}_{\lambda_1,\lambda_2,\lambda_3}$  such that
\begin{center}
$\Delta(x)=D_{x,e_i}(x), \hspace{0.2 cm} \Delta(e_i)=D_{x,e_i}(e_i)$.
\end{center}
By Theorem \ref{thm1}, the matrix $M^{x,e_i}$ of the derivation $D_{x,e_i}$ has the following matrix form:
$$M^{x,e_i}=\begin{pmatrix}
		0 & 0 &0 &0\\
		0&0&\frac{-\lambda_3}{\lambda_2}a^{x,e_i} & -\frac{\lambda_3}{\lambda_1} b^{x,e_i} \\
		0 & a^{x,e_i} & 0 & -\frac{\lambda_2}{\lambda_1} c^{x,e_i} \\
		0 & b^{x,e_i} & c^{x,e_i} & 0  
\end{pmatrix}.$$
Let $y \in \mathcal{H}_d$. Then there exists a derivation $D_{y,e_i}$ of $\mathcal{H}_d$ such that 
\begin{center}
$\Delta(y)=D_{y,e_i}(y), \hspace{0.3 cm} \Delta(e_i)=D_{y,e_i}(e_i).$
\end{center}
By Theorem \ref{thm1}, the matrix $M^{y,e_i}$ of the derivation $D_{y,e_i}$ has the following matrix form:
$$M^{y,e_i}=\begin{pmatrix}
		0 & 0 &0 &0\\
		0&0&\frac{-\lambda_3}{\lambda_2}a^{y,e_i} & -\frac{\lambda_3}{\lambda_1} b^{y,e_i} \\
		0 & a^{y,e_i} & 0 & -\frac{\lambda_2}{\lambda_1} c^{y,e_i} \\
		0 & b^{y,e_i} & c^{y,e_i} & 0   
		\end{pmatrix}.$$
Since $\Delta(e_i)=D_{x,e_i}(e_i)=D_{y,e_i}(e_i)$, we have
\begin{gather*}
		  	 \begin{cases}
		  	\Delta(e_0)=0 ;&\\
		  	\,a^{x,e_i}=a^{y,e_i} ; \\
		  	\,b^{x,e_i}=b^{y,e_i}  ;&\\	
		  	\, c^{x,e_i}=c^{y,e_i}  .&\\ 	  	 
		  	 \end{cases}
\end{gather*}
That it $$M^{x,e_i}=M^{y,e_i}.$$
Therefore, for any $x \in \mathcal{H}_d$
$$\Delta(x)=D_{y,e_i}(x),$$
that it $D_{x,e_i}$ does not depend on x. Hence, $\Delta$ is a derivation by Theorem \ref{thm1}.
\end{proof}
\section{ Biderivations of  3-parameter generalized quaternion}
In this section, we determine all biderivations of  3-parameter generalized quaternion. 
\begin{definition}
\label{def5}
 Let $\mathcal{A}$ be an arbitrary algebra. A bilinear map $\delta : \mathcal{A} \times \mathcal{A} \rightarrow \mathcal{A}$ is called a biderivation on $\mathcal{A}$ if
 $$\delta(xy,z)=x\delta(y,z)+\delta(x,z)y,$$
 $$\delta(x,yz)=y\delta(x,z)+\delta(x,y)z,$$
 for all $x,y,z \in \mathcal{A}$.
 \end{definition}
  Denote by $BDer(\mathcal{A})$ the set of all biderivations on $\mathcal{A}$ which is clearly a vector space.\\
  An example of a biderivation of $\mathbb{K}_{\lambda_1,\lambda_2,\lambda_3}$ is the map 
  $$\delta: \mathbb{K}_{\lambda_1,\lambda_2,\lambda_3}\times \mathbb{K}_{\lambda_1,\lambda_2,\lambda_3} \to \mathbb{K}_{\lambda_1,\lambda_2,\lambda_3}, (x_0+\tilde{x},y_0+\tilde{y}) \mapsto \delta(x_0+\tilde{x},y_0+\tilde{y}):=\mu \tilde{x} \wedge \tilde{y}.$$
A $\delta \in BDer(\mathcal{A})$ is called symmetric if $\delta(x,y)=\delta(y,x)$ for all $x,y \in \mathcal{A}$, and is called skew-symmetric if $\delta(x,y)=-\delta(y,x)$ for all $x,y \in \mathcal{A}$. Denote by $BDer_+(\mathcal{A})$ and $BDer_-(\mathcal{A})$ the subspace of all symmetric biderivations and all skew-symmetric biderivations on $\mathcal{A}$ respectively.\\
For any $\delta \in BDer(\mathcal{A})$, we define two bilinear maps by 
$$\delta^+(x,y)=\delta(x,y)+\delta(y,x), \hspace{0.3 cm}\delta^-(x,y)=\delta(x,y)-\delta(y,x).$$
It is easy to see that $\delta^+ \in BDer_+(\mathcal{A})$ and $\delta^- \in BDer_-(\mathcal{A})$. Since $\delta =\frac{1}{2}(\delta^+ + \delta^-)$ it follows that
$$BDer(\mathcal{A})=BDer_+(\mathcal{A}) \oplus BDer_-(\mathcal{A}).$$
To characterize $BDer(\mathcal{A})$, we only need to characterize $BDer_+(\mathcal{A})$ and $BDer_-(\mathcal{A})$.\\
Now we give a  description of derivations on 3-parameter generalized quaternion $\mathbb{K}_{\lambda_1,\lambda_2, \lambda_3}$.\\
For any $x=x_0+\tilde{x}\in \mathbb{K}_{\lambda_1,\lambda_2, \lambda_3} $ we denote by $ad_{\wedge,\tilde{x}}$ the left multiplication by $x$ given by $$ad_{\wedge,x}y=\tilde{x}\wedge\tilde{y}=\begin{vmatrix}
		\frac{\lambda_3}{\lambda_1} e_1&\frac{\lambda_2}{\lambda_1} e_2&e_3\\
		x_1&x_2&x_3 \\
		y_1&y_2&y_3
		\end{vmatrix}$$ for any $y=y_0+\tilde{y}\in \mathbb{K}_{\lambda_1,\lambda_2, \lambda_3}$. We have the following lemma
\begin{lemma}\label{lem1}
    Let $D$ be a derivation of $\mathbb{K}_{\lambda_1,\lambda_2, \lambda_3}$ $(\lambda_3\neq 0)$. Then, there exists $x\in \mathbb{K}_{\lambda_1,\lambda_2, \lambda_3}$ such that $D=ad_{\wedge,x}$.
\end{lemma}
\begin{proof}
 By a straightforward calculation we prove that $ad_{\wedge,x}$ is a derivation of $\mathbb{K}_{\lambda_1,\lambda_2, \lambda_3}$ and 
 \begin{eqnarray*}
     &&ad_{\wedge,e_0}=\begin{pmatrix}
		0 & 0 &0 &0\\
		0&0&0 & 0 \\
		0 & 0 & 0 & 0 \\
		0 & 0 & 0 & 0  
\end{pmatrix} \quad ad_{\wedge,e_1}=\begin{pmatrix}
		0 & 0 &0 &0\\
		0&0&0 & 0 \\
		0 & 0 & 0 & -\frac{\lambda_2}{\lambda_1}  \\
		0 & 0 & 1 & 0  
\end{pmatrix}\\
    && ad_{\wedge,e_2}= \begin{pmatrix}
		0 & 0 &0 &0\\
		0&0&0 & \frac{\lambda_3}{\lambda_1}  \\
		0 & 0 & 0 & 0 \\
		0 & -1 & 0 & 0  
\end{pmatrix}\quad ad_{\wedge,e_3}=\begin{pmatrix}
		0 & 0 &0 &0\\
		0&0&\frac{\lambda_3}{\lambda_1} &0 \\
		0 & -\frac{\lambda_2}{\lambda_1} & 0 & 0 \\
		0 & 0 & 0 & 0  
\end{pmatrix} .
 \end{eqnarray*}
 In the other hand, by Theorem \ref{thm1}, $dim(Der(\mathbb{K}_{\lambda_1,\lambda_2,\lambda}))=3$. Therefor, $\lbrace ad_{\wedge,e_1}, ad_{\wedge,e_2}, ad_{\wedge,e_3} \rbrace$ form a basis of $dim(Der(\mathbb{K}_{\lambda_1,\lambda_2,\lambda}))=3$. This complete the proof.
\end{proof}
\begin{lemma}\label{lem2}
    Let $\delta$ be a biderivation of $\mathbb{K}_{\lambda_1,\lambda_2, \lambda_3}$ ($\lambda_3\neq 0$). Then, there exist two linear map $\mathbb{K}_{\lambda_1,\lambda_2, \lambda_3} \to span\lbrace e_1,e_2,e_3\rbrace$ such that 
    $$\delta(w,y)=\tilde{\phi(x)}\wedge\tilde{y} =\tilde{x}\wedge \tilde{\psi(y)}, \quad \forall x,y\in \mathbb{K}_{\lambda_1,\lambda_2, \lambda_3}.$$
\end{lemma}
\begin{proof}
Four a fixed element $x\in \mathbb{K}_{\lambda_1,\lambda_2, \lambda_3}$, we define a map $\delta_x :\mathbb{K}_{\lambda_1,\lambda_2, \lambda_3} \to \mathbb{K}_{\lambda_1,\lambda_2, \lambda_3}$ by
$$\delta_x(y)=\delta(x,y), \quad \forall y\in \mathbb{K}_{\lambda_1,\lambda_2, \lambda_3}.$$
Since $\delta$ is a biderivation, then $\delta_x$ is a derivation. By Lemma \ref{lem1} there exists an element $\tilde{z} \in span\lbrace e_1,e_2,e_3\rbrace$ such that $\delta_x=ad_{\wedge,\tilde{z}}$. In fact, such $\tilde{z}$ is unique. Otherewise, if there exists an element $\tilde{z'}\in span\lbrace e_1,e_2, e_3\rbrace$ such that $ad_{\wedge,\tilde{z}}=ad_{\wedge,\tilde{z'}}$. Then for any $y\in \mathbb{K}_{\lambda_1,\lambda_2, \lambda_3}$, we have $\tilde{z'}\wedge \tilde{y}=ad_{\wedge,\tilde{z'}}(y)=ad_{\wedge,\tilde{z}}(y)=\tilde{z}\wedge \tilde{y}$, so $\tilde{z'}=\tilde{y}$. Thus we can define a map $\phi:\mathbb{K}_{\lambda_1,\lambda_2, \lambda_3} \to span\lbrace e_1,e_2,e_3\rbrace$ such that $\delta_x=ad_{\wedge,\phi(x)}$.i.e $\delta(x,y)=\phi(x)\wedge \tilde{y}$. Furthermore, the bilinearity of $\delta$ implies that $\phi$ is linear. Similarly, we define a map $\delta_z$ from $\mathbb{K}_{\lambda_1,\lambda_2, \lambda_3}$ into it self given by $\delta_z(y)=\delta(y,z)$ for all $y\in \mathbb{K}_{\lambda_1,\lambda_2, \lambda_3}$. We can obtain a linear $\psi$ from $\mathbb{K}_{\lambda_1,\lambda_2, \lambda_3}$ into $span\lbrace e_1,e_2,e_3\rbrace$ such that $\delta(y,z)=ad_{\wedge,-\psi(y)}\tilde{x}\wedge\psi(y)$. The proof is completed.
\end{proof}
\begin{lemma}\label{lem3}
    Let $\phi,\psi$ be defined by Lemma \ref{lem2}. Then, there exists a real number $\mu$ such that 
    $$\phi(e_i)=\psi(e_i)=\mu e_i, \quad \forall i=1,2,3.$$
\end{lemma}
\begin{proof}
    For any $i=1,2,3$. By letting
    $$\phi(e_i)=\sum_{j=1}^3a_{ji}e_j \quad \mbox{and}\quad \psi(e_i)=\sum_{j=1}^3b_{ji}e_j.$$
 The equality $\phi(e_i) \wedge e_i = e_i \wedge \psi(e_i)$ imply that
\[
\begin{cases}
    a_{21} = -b_{21}, & a_{31} = -b_{31}; \\
    a_{12} = -b_{12}, & a_{32} = -b_{32}; \\
    a_{13} = -b_{13}, & a_{23} = -b_{23}.
\end{cases}
\]
The equality $\phi(e_1) \wedge e_2 = e_1 \wedge \psi(e_2)$ imply that
$$a_{11}=b_{22}\quad \mbox{and}\quad a_{31}=b_{32}=0.$$
The equality $\phi(e_1) \wedge e_3 = e_1 \wedge \psi(e_3)$ imply that
$$a_{11}=b_{33} \quad \mbox{and} \quad a_{21}=b_{23}=0.$$
The equality $\phi(e_2) \wedge e_3 = e_2 \wedge \psi(e_3)$ imply that
$$a_{22}=b_{33}\quad \mbox{and}\quad a_{12}=b_{13}=0.$$
The equality $\phi(e_3) \wedge e_1 = e_3 \wedge \psi(e_1)$ imply that
$$a_{33}=b_{11} \quad \mbox{and}\quad a_{23}=b_{21}=0.$$
The equality $\phi(e_3) \wedge e_2 = e_3 \wedge \psi(e_2)$ imply that
$$a_{33}=b_{22} \quad \mbox{and}\quad a_{13}=b_{12}=0.$$
By setting $a_{ii} = b_{ii} = \mu$, the proof follows from end. 
\end{proof}
We now state our main result of this section as follows.
\begin{theorem}\label{thmb}
 Let $\mathbb{K}_{\lambda_1,\lambda_2, \lambda_3}$ be  the 3-parameter generalized quaternion algebra. Then $\delta$ is a biderivation if and only if there is a real number $\mu$ such that
 $$\delta(x,y)=\mu \tilde{x}\wedge\tilde{y}, \quad \forall x,y \in \mathbb{K}_{\lambda_1,\lambda_2, \lambda_3}. $$
\end{theorem}
\begin{proof}
    The "if" direction is obvious. We now prove the "onlyif" direction. We first see that $\delta(e_0,x)=\delta(x,e_0)=\delta_x(e_0)=0$ for all $x\in \mathbb{K}_{\lambda_1,\lambda_2, \lambda_3}$. By Lemmas \ref{lem1}, \ref{lem2} and Lemma \ref{lem3} there exists $\mu \in\mathbb{R}$ such that 
    $$\delta(x,y)=\delta(\sum_{i=0}^3x_ie_i,y)=\sum_{i=1}^3x_i\phi(e_i)\wedge\tilde{y}=\sum_{i=1}^3\mu x_ie_i\wedge\tilde{y}=\mu \tilde{x}\wedge\tilde{y}.$$
\end{proof}

%\textbf{Case 1} Skew-symmetric biderivations. \\
Now, we go to characterize all biderivation on the case $\lambda_3=0$.\\
Let $\delta$ be a  biderivation on  3-parameter generalized quaternion $\mathbb{K}_{\lambda_1,\lambda_2, 0}$  and $x,y \in \mathbb{K}_{\lambda_1,\lambda_2, 0}$, such that $x=\sum_{i=0}^{3}x_{i}e_{i}$ and $y=\sum_{i=0}^{3} y_{i}e_{i}$. Then, by the bilinearity of $\delta$, we obtain, 
$$\delta(x,y)=\sum_{i=0}^{3} \sum_{j=0}^{3} x_{i} y_{j} \delta(e_i,e_j)=
\sum_{i=0}^{3} \sum_{j=0}^{3} x_{i} y_{j} \delta_{e_i}(e_j).$$ 
By Theorem \ref{thm1}, the matrix $D_{e_i}$ of $\delta_{e_i}$, for $i=0,1,2,3$ in the basis $\lbrace e_0,e_1,e_2,e_3 \rbrace$ is of the form
\begin{equation}
\label{eq2}
D_{e_i}= \begin{pmatrix}
		0 & 0 &0 &0\\
		0&0&0 & 0 \\
		0 & a_i & d_i & -\frac{\lambda_2}{\lambda_1} c_i \\
		0 & b_i & c_i & d_i  
		\end{pmatrix}. 
\end{equation}
\textbf{cases:1} $\delta$ is skew-symmetric.\\
From the equalities $\delta(e_i,e_i)=0$ for $i=1,2,3$, we have 
$$a_1=b_1=d_2=c_2=c_3=d_3=0.$$
From the equalities $\delta(e_0,e_i)=-\delta(e_i,e_0)=0$ for $i=1,2,3$, we have 
$$a_0=b_0=c_0=d_0=0.$$
From the equalities $\delta(e_1,e_i)=-\delta(e_i,e_1)$ for $i=2,3$, we get 
$$d_1=-a_2, \quad c_1=-b_2,\quad \frac{\lambda_2}{\lambda_1}c_1=a_3, \quad d_1=-b_3.$$
From the equalities $\delta(e_1,e_i)=-\delta(e_i,e_1)$ for $i=2,3$, we obtain
$$\frac{\lambda_2}{\lambda_1}c_2=d_3, \quad d_2=-c_3.$$
Therefore, by setting $a=c_1$ and $b=d_1$ we deduce that
\begin{eqnarray*}
    &&D_{e_0}= \begin{pmatrix}
		0 & 0 &0 &0\\
		0&0&0 & 0 \\
		0 & 0 & 0 & 0 \\
		0 & 0 & 0 &   
		\end{pmatrix}, \quad D_{e_1}= \begin{pmatrix}
		0 & 0 &0 &0\\
		0&0&0 & 0 \\
		0 & 0 & b & -\frac{\lambda_2}{\lambda_1} a \\
		0 & 0 & a & b  
		\end{pmatrix},\\
        &&D_{e_2}= \begin{pmatrix}
		0 & 0 &0 &0\\
		0&0&0 & 0 \\
		0 & -b & 0 & 0 \\
		0 & -a & 0 & 0  
		\end{pmatrix}, \quad D_{e_i}= \begin{pmatrix}
		0 & 0 &0 &0\\
		0&0&0 & 0 \\
		0 & \frac{\lambda_2}{\lambda_1} a& 0 & 0 \\
		0 & -b & 0 & 0  
		\end{pmatrix}
\end{eqnarray*}
Then, $\delta(x,y)=x_1D_{e_1}Y+x_2D_{e_2}Y+x_3D_{e_3}Y$, for all $x,y\in \mathbb{K}_{\lambda_1,\lambda_2, 0}$ where $Y=(y_0,y_1,y_2,y_3)^T$.\\
Therefore, we have the following theorem
\begin{theorem}
\label{thm4}
Let $\mathbb{K}_{\lambda_1,\lambda_2 , 0 }$  be a 3-parameter generalized quaternion. Then, $\delta^-$ is a skew-symmetric biderivation on $\mathbb{K}_{\lambda_1,\lambda_2 , 0 }$ if and only if there exist two real numbers $a,b$ such that 
$$\delta^-(x,y)=q \begin{vmatrix}
		0&\frac{\lambda_2}{\lambda_1} e_2&e_3\\
		x_1&x_2&x_3 \\
		y_1&y_2&y_3
		\end{vmatrix}+b\begin{vmatrix}
		0&- e_3&e_2\\
		x_1&x_2&x_3 \\
		y_1&y_2&y_3
		\end{vmatrix}, \hspace{0.3 cm}\forall x=\sum_{i=0}^3x_ie_i,\hspace{0.2 cm} y=\sum_{i=0}^3y_ie_i \in \mathbb{K}_{\lambda_1,\lambda_2 , 0 }.$$
\end{theorem}
$\textbf{cases: 2}$ $\delta$ is symmetric.\\
Similarly to the case 1, we use the fact that $\delta$ is symmetric we have
\begin{eqnarray*}
    &&D_{e_0}= \begin{pmatrix}
		0 & 0 &0 &0\\
		0&0&0 & 0 \\
		0 & 0 & 0 & 0 \\
		0 & 0 & 0 & 0  
		\end{pmatrix}, \quad D_{e_i}= \begin{pmatrix}
		0 & 0 &0 &0\\
		0&0&0 & 0 \\
		0 & a_1 & d_1 & -\frac{\lambda_2}{\lambda_1} c_1 \\
		0 & b_1 & c_1 & d_1  
		\end{pmatrix},\\
        &&D_{e_2}= \begin{pmatrix}
		0 & 0 &0 &0\\
		0&0&0 & 0 \\
		0 & d_1 & d_2 & -\frac{\lambda_2}{\lambda_1} c_2 \\
		0 & c_1 & c_2 & d_2  
		\end{pmatrix},\quad D_{e_3}= \begin{pmatrix}
		0 & 0 &0 &0\\
		0&0&0 & 0 \\
		0 & -\frac{\lambda_2}{\lambda1}c_1 & -\frac{\lambda_2}{\lambda_1}c_2 & -\frac{\lambda_2}{\lambda_1} d_2 \\
		0 & d_1 & d_2 & -\frac{\lambda_2}{\lambda_1}c_2  
		\end{pmatrix}
\end{eqnarray*}
and $\delta(x,y)=x_1D_{e_1}Y+x_2D_{e_2}Y+x_3D_{e_3}Y$, for all  for all $x,y\in \mathbb{K}_{\lambda_1,\lambda_2, 0}$.

\section{Commuting linear maps and centroid on a 3-parameter generalized quaternion}
The centroid of algebras plays an important role in understanding the structure of algebras. All scalar extensions of a simple algebra remain simple if and only if its cenroid just consists of the scalars in the base field. In particular, for finite-dimensional simple associative algebras, the centroid is critical in investigating Brauer groups and division algebras. Another area where the centroid occurs naturally is in the study of derivations of an algebra. The centroids of Lie algebras have been studied in \cite{DM}. It is well known that the centroid of a simple Lie algebra is a field, This plays an important role in the classification problem of finite dimensional extended affine Lie algebras over arbitrary field of characteristic zero \cite{BN}. In this paper we prove that the centroid of a 3-parameter generalized quaternion over  $\mathbb{R}$ is a filed.
\begin{definition}
\label{def6}
A commuting linear map of an algebra  $\mathcal{A}$  is a linear map $\phi : \mathcal{A} \rightarrow \mathcal{A}$ such that
$$ \phi(x).x=x.\phi(x), \hspace{0.3 cm} \forall x \in \mathcal{A}.$$  
\end{definition}
\begin{definition}
\label{def7}
Let $\mathbb{K}_{\lambda_1,\lambda_2, \lambda_3}$ be a 3-parameter generalized quaternion. The set 
$$\Gamma(\mathbb{K}_{\lambda_1,\lambda_2, \lambda_3})=\left\lbrace  \gamma \in End(\mathbb{K}_{\lambda_1,\lambda_2, \lambda_3}) \hspace{0.2 cm} : \hspace{0.2 cm} \gamma(x.y)=\gamma(x).y=x.\gamma(y), \hspace{0.2 cm} \forall x,y \in \mathbb{K}_{\lambda_1,\lambda_2, \lambda_3} \right\rbrace $$
is called the cenroid of $\mathbb{K}_{\lambda_1,\lambda_2, \lambda_3}$.
\end{definition}
Also we can define the set
$$ Q\Gamma(\mathbb{K}_{\lambda_1,\lambda_2, \lambda_3})=\lbrace \gamma \in End(\mathbb{K}_{\lambda_1,\lambda_2, \lambda_3}) \hspace{0.2 cm}: \hspace{0.2 cm} \gamma(x).y=x.\gamma(y), \hspace{0.2 cm} \forall x,y \in \mathbb{K}_{\lambda_1,\lambda_2, \lambda_3} \rbrace,$$
which is called the quasi-centroid of a 3-parameter generalized quaternion $\mathbb{K}_{\lambda_1,\lambda_2, \lambda_3}$.\\
 It is a simple observation to see that $\Gamma(\mathbb{K}_{\lambda_1,\lambda_2, \lambda_3})$ is closed under composition and thus has Lie algebra structure. Hence, we can consider $\Gamma(\mathbb{K}_{\lambda_1,\lambda_2, \lambda_3})$ as Lie subalgebra of $End(\mathbb{K}_{\lambda_1,\lambda_2, \lambda_3})$.
\begin{lemma}
\label{lem1}
Let $\mathbb{K}_{\lambda_1,\lambda_2, \lambda_3}$ be a generalized quaternion algebra and $\phi \in \Gamma(\mathbb{K}_{\lambda_1,\lambda_2, \lambda_3})$, $d \in Der(\mathbb{K}_{\lambda_1,\lambda_2, \lambda_3})$. Then \\
$i)\hspace{0.2 cm} \phi \circ d \in Der(\mathbb{K}_{\lambda_1,\lambda_2, \lambda_3}). $\\
$ii)\hspace{0.2 cm} [\phi ,d] \in \Gamma(\mathbb{K}_{\lambda_1,\lambda_2, \lambda_3}). $
\end{lemma}
\begin{proof}
 Let $\phi \in \Gamma(\mathbb{K}_{\lambda_1,\lambda_2, \lambda_3})$ and $d \in Der(\mathbb{K}_{\lambda_1,\lambda_2, \lambda_3})$\\
$i)$ Let $x,y \in \mathbb{K}_{\lambda_1,\lambda_2, \lambda_3}$, we have
$$\phi \circ d(x.y)=\phi(d(x).y+x.d(y))=\phi(d(x)).y+x.\phi(d(y))$$
Then $\phi \circ d \in Der(\mathbb{K}_{\lambda_1,\lambda_2, \lambda_3})$.\\
$ii)$  Let $x,y \in \mathbb{K}_{\lambda_1,\lambda_2, \lambda_3}$, we have 
$$[\phi ,d](x.y)=\phi \circ d (x.y)-d \circ \phi(x.y)=\phi(d(x).y+x.d(y))-d(\phi(x).y)$$
$$=\phi(d(x)).y+\phi(x).d(y)-d(\phi(x)).y-\phi(x).d(y)$$
$$=\phi(d(x)).y-d(\phi(x)).y=(\phi(d(x))-d(\phi(x))).y=([\phi, d](x)).y.$$
Similarly  $[\phi ,d](x.y)=x.([\phi, d](y))$. So $ [\phi ,d] \in \Gamma(\mathbb{K}_{\lambda_1,\lambda_2, \lambda_3})$. 
\end{proof} 
We state our second main result as follows.
\begin{theorem}
\label{thm6}
Let $\mathbb{K}_{\lambda_1,\lambda_2, \lambda_3}$ be a generalized quaternion algebra. Then, the matrix of any commuting linear map on $\mathbb{K}_{\lambda_1,\lambda_2, \lambda_3}$ is of the form
$$
\begin{pmatrix}
\lambda & a & b & c \\
0& \mu & 0 & 0 \\
0 &0 & \mu &0 \\
0&0&0&\mu
\end{pmatrix} \in M_4(\mathbb{R})
$$
\end{theorem}
\begin{proof}
Let $\phi$ be a commuting linear map of $\mathbb{K}_{\lambda_1,\lambda_2, \lambda_3}$, and $M=(a_{ij})_{0\leq i,j \leq 3}$ its matrix  in the basis $\lbrace e_0,e_1,e_2,e_3 \rbrace$. \\
We define the bilinear map $\delta:\mathbb{K}_{\lambda_1,\lambda_2, \lambda_3}\times \mathbb{K}_{\lambda_1,\lambda_2, \lambda_3} \to \mathbb{K}_{\lambda_1,\lambda_2, \lambda_3}$ by $\delta(x,y)=\tilde{\phi(x)}\wedge\tilde{y}$ it is easy to see that $\delta$ is a biderivation on $\mathbb{K}_{\lambda_1,\lambda_2, \lambda_3}$. Then, by Theorem \ref{thmb}, there exists a real numbre $\mu$ such that $\delta(x,y)=\mu \tilde{x}\wedge \tilde{y}$. That is $\tilde{\phi(x)}=\mu\tilde{x}$. Therefore, the matrix of has the form
$$M=\begin{pmatrix}
\lambda & a & b & c \\
0& \mu & 0 & 0 \\
0 &0 & \mu &0 \\
0&0&0&\mu
\end{pmatrix} \in M_3(\mathbb{R}),
$$ where $\lambda =a_{00},$ $a=a_{01}$, $b=a_{02}$, $c=a_{03}$ and $\mu=a_{11}=a_{22}=a_{33}$.
\end{proof}
\begin{theorem}
\label{thm7}
Let $\mathbb{K}_{\lambda_1,\lambda_2, \lambda_3}$ be a generalized quaternion algebra. Then, 
$$Q\Gamma(\mathbb{K}_{\lambda_1,\lambda_2, \lambda_3})=\left\lbrace \lambda id \hspace{0.3 cm}: \hspace{0.2 cm} \lambda \in \mathbb{R} \right\rbrace . $$
\end{theorem} 
\begin{proof}
Let $\phi  \in Q\Gamma(\mathbb{K}_{\lambda_1,\lambda_2, \lambda_3})$, then $\phi$ is a commuting linear map, so the matrix of $\phi$ is given by Theorem \ref{thm6}. Since $\phi(x).y=x.\phi(y)$, xe have that,\\
The equality $\phi(e_1).e_0 =e_1.\phi(e_0)$  implies that 
$$ a=0 \hspace{0.1 cm} and \hspace{0.1 cm} \mu = \lambda .$$
The equality $\phi(e_2).e_0 =e_2.\phi(e_0)$ implies that
 $$ b=0 \hspace{0.1 cm} and \hspace{0.1 cm} \mu = \lambda.$$
The equality $\phi(e_3).e_0 =e_3.\phi(e_0)$ implies that
$$ c=0 \hspace{0.1 cm} and \hspace{0.1 cm} \mu = \lambda .$$
Finally $\phi(e_i)=\lambda e_i$ for $i=0,1,2,3$.
\end{proof}
We now state our last result as follows.
\begin{theorem}
\label{thm8}
Let $\mathbb{K}_{\lambda_1,\lambda_2, \lambda_3}$ be a generalized quaternion algebra. Then, 
$$\Gamma(\mathbb{K}_{\lambda_1,\lambda_2, \lambda_3})=\left\lbrace \lambda id \hspace{0.3 cm}: \hspace{0.2 cm} \lambda \in \mathbb{R} \right\rbrace . $$
\end{theorem}
\begin{proof}
It is clear that the scalar map belongs in $\Gamma(\mathbb{K}_{\lambda_1,\lambda_2, \lambda_3})$, therefore $\Gamma(\mathbb{K}_{\lambda_1,\lambda_2, \lambda_3}) \subseteq Q\Gamma(\mathbb{K}_{\lambda_1,\lambda_2, \lambda_3})$. then, the result follows from Theorem \ref{thm7}. 
\end{proof}
\vspace{1cm}
\noindent \textbf{Acknowledgements:} The authors thank the referees for their valuable comments that contributed to a sensible improvement of the paper.

\section*{Declarations}
\noindent \textbf{Conflicts of interest:} The authors declare that they have no conflict of interest.\\
\noindent \textbf{Funding:} Not applicable.\\
\noindent \textbf{Ethical approval:} Not applicable.\\
\noindent \textbf{Data availibility:} Consent for publication.

\end{document}